\documentclass[12pt,oneside,english]{amsart}
\usepackage[T1]{fontenc}
\usepackage[latin9]{inputenc}
\usepackage{geometry}
\usepackage{amstext}
\usepackage{amsthm}
\usepackage{amssymb}

\makeatletter
\numberwithin{equation}{section}
\numberwithin{figure}{section}
\theoremstyle{plain}
\newtheorem{thm}{\protect\theoremname}
  \theoremstyle{plain}
  \newtheorem{cor}{\protect\corollaryname}
  \theoremstyle{plain}
  \newtheorem{lem}{\protect\lemmaname}
  \theoremstyle{definition}
  \newtheorem{defn}{\protect\definitionname}
  \theoremstyle{remark}
  \newtheorem{rem}{\protect\remarkname}
  \theoremstyle{plain}
  \newtheorem{prop}{\protect\propositionname}

\DeclareMathOperator{\slicerank}{slice-rank}
\DeclareMathOperator{\sgn}{sgn}
\DeclareMathOperator{\prank}{partition-rank}

\makeatother

\usepackage{babel}
  \providecommand{\definitionname}{Definition}
  \providecommand{\lemmaname}{Lemma}
  \providecommand{\propositionname}{Proposition}
  \providecommand{\remarkname}{Remark}
\providecommand{\corollaryname}{Corollary}
\providecommand{\theoremname}{Theorem}

\usepackage{hyperref}
\providecommand{\MR}{\relax\ifhmode\unskip\space\fi MR }
\providecommand{\MRhref}[2]{%
  \href{http://www.ams.org/mathscinet-getitem?mr=#1}{#2}
}

\newcommand\arXiv[1]{arXiv:\href{http://arXiv.org/abs/#1}{#1}}

\begin{document}

\title{Exponential Bounds for the Erd\H{o}s-Ginzburg-Ziv Constant}

\author{Eric Naslund}

\date{\today}
\begin{abstract}
The Erd\H{o}s-Ginzburg-Ziv constant of a finite abelian group $G$, denoted
$\mathfrak{s}(G)$, is the smallest $k\in\mathbb{N}$ such that any
sequence of elements of $G$ of length $k$ contains a zero-sum subsequence
of length $\exp(G)$. In this paper, we use the partition rank from
\cite{Naslund2017Partition}, which generalizes the slice rank, to
prove that for any odd prime $p$, 
\[
\mathfrak{s}\left(\mathbb{F}_{p}^{n}\right)\leq(p-1)2^{p}\left(J(p)\cdot p\right)^{n}
\]
where $0.8414<J(p)<0.91837$. For large $n$, and $p>3$, this
is the first exponential improvement to the trivial bound. We also
provide a near optimal result conditional on the conjecture that $\left(\mathbb{Z}/k\mathbb{Z}\right)^{n}$
satisfies \emph{property $D$,} as defined in \cite{GaoGeroldinger2006ZeroSumSurvey},
showing that in this case
\[
\mathfrak{s}\left(\left(\mathbb{Z}/k\mathbb{Z}\right)^{n}\right)\leq(k-1)4^{n}+k.
\]
\end{abstract}

\maketitle

\section{Introduction}

For an abelian group $G$, let $\exp(G)$ denote the exponent of $G$,
which is the maximal order of any element in $G$. The Erd\H{o}s-Ginzburg-Ziv
constant of $G$, denoted $\mathfrak{s}(G)$, is defined to be the
smallest $k\in\mathbb{N}$ such that any sequence of elements from
$G$ of length $k$ contains a zero-sum subsequence of length $\exp(G)$.
In 1961, Erd\H{o}s, Ginzburg, and Ziv \cite{ErdosGinzburgZiv1961Original}
 proved that 
\[
\mathfrak{s}\left(\mathbb{Z}/k\mathbb{Z}\right)=2k-1.
\]
That is, among any sequence of $2k-1$ integers there is a subsequence
of length $k$ which sums to $0$ in $\left(\mathbb{Z}/k\mathbb{Z}\right)^n$, and furthermore this is not true
if $2k-1$ is replaced by $2k-2$. In 2007, Reiher \cite{Reiher2007KemnitzConj}
resolved a longstanding conjecture of Kemnitz \cite{Kemnitz1983Conj},
and proved that 
\[
\mathfrak{s}\left(\left(\mathbb{Z}/k\mathbb{Z}\right)^{2}\right)=4k-3.
\]
When $G$ is a power of a cyclic group, $\mathfrak{s}(G)$ has a geometric
interpretation: $\mathfrak{s}((\mathbb{Z}/k\mathbb{Z})^{n})$ is the
smallest integer such that any set of $\mathfrak{s}((\mathbb{Z}/k\mathbb{Z})^{n})$
points in $\mathbb{Z}^{n}$ contains $k$ points whose average is
again a lattice point. When $G=\left(\mathbb{Z}/k\mathbb{Z}\right)^{n}$
with $n$ large, very little is known. Harborth \cite{Harborth1973zerosumsets}
gave the elementary bounds 
\[
(k-1)2^{n}+1\leq\mathfrak{s}\left(\left(\mathbb{Z}/k\mathbb{Z}\right)^{n}\right)\leq(k-1)k^{n}+1,
\]
and Elsholtz \cite{Elsholtz2003EGZLowerBounds} improved the lower
bound for odd $k$ to 
\[
\mathfrak{s}\left(\left(\mathbb{Z}/k\mathbb{Z}\right)^{n}\right)\geq(k-1)2^{n}\cdot\left(\frac{9}{8}\right)^{\left[\frac{n}{3}\right]}+1 = (k-1)(2.08+o(1))^{n}.
\]
In a different direction, Alon and Dubiner \cite{AlonDubiner1993Zerosumsets}
proved that 
\[
\mathfrak{s}\left(\left(\mathbb{Z}/k\mathbb{Z}\right)^{n}\right)\leq\left(cn\log n\right)^{n}k
\]
for some $c>0$, which implies that for fixed $n$, $\mathfrak{s}\left(\left(\mathbb{Z}/k\mathbb{Z}\right)^{n}\right)$
grows linearly in $k$. In this paper, we give a conditional improvement
to Alon and Dubiner's bound, that is within a factor of $2^{n}$ of
the lower bound, and we give an unconditional exponential improvement
to the upper bounds for $\mathfrak{s}\left(\left(\mathbb{Z}/p\mathbb{Z}\right)^{n}\right)$
for large $n$. 

Let $J(p)$ be the constant given by the minimization 
\[
J(p)=\frac{1}{p}\min_{0<x<1}\frac{1-x^{p}}{1-x}x^{-\frac{p-1}{3}}.
\] 
In \cite[prop. 4.12]{BlasiakChurchCohnGrochowNaslundSawinUmans2016MatrixMultiplication}
it was shown that $J(p)$ is a decreasing function of $p$ that satisfies
\[
J(3)=\frac{1}{8}\sqrt[3]{207+33\sqrt{33}}=0.9183\dots,
\]
and 
\[
\lim_{p\rightarrow\infty}J(p)=\inf_{z>1}\frac{z-z^{-2}}{3\log z}=0.8414\dots.
\]
and so $0.8414\leq J(p)\leq0.918$ for $p\geq 3$. Our main result is:
\begin{thm}
\label{thm:Main_Theorem} Let $p>2$ be a prime. Suppose that $A\subset\mathbb{F}_{p}^{n}$
does not contain $p$ distinct elements $x_{1},\dots,x_{p}$ such
that 
\[
x_{1}+\cdots+x_{p}=0.
\]
Then
\[
|A|\leq(2^{p}-p-2)(pJ(p))^{n}
\]
and
\[
\mathfrak{s}\left(\mathbb{F}_{p}^{n}\right)\leq(p-1)2^{p}(pJ(p))^{n}.
\]
\end{thm}
Note that the case $p=3$ is a consequence of Ellenberg and Gijswijt's
Theorem \cite{EllenbergGijswijtCapsets}. The bound for $\mathfrak{s}(\mathbb{F}_{p}^{n})$
has since been improved to $2p\cdot(pJ(p))^{n}$ by Fox and Sauermann
\cite{FoxSauermann2017EGZ}. Furthermore, they proved that bounding $\mathfrak{s}\left((\mathbb{Z}/k\mathbb{Z})^{n}\right)$
for general $k$ can be reduced to bounding $\mathfrak{s}(\mathbb{F}_{p}^{n})$
for $p|k$, and we will reproduce this argument in section \ref{sec:Reduction-to-Prime}. This reduction argument was first communicated to the author by Sauermann before their paper appeared, and using this reduction, we have the
following Corollary:
\begin{cor}
\label{cor: cp_asymp}Let $q$ be the largest prime power dividing
$k$. Then 
\[
\mathfrak{s}\left(\left(\mathbb{Z}/k\mathbb{Z}\right)^{n}\right)\leq\left(J(q)q+o(1)\right)^{n},
\]
 and the $o(1)$ term tends to $0$ as $n\rightarrow\infty$. 
\end{cor}

This Corollary was subsequently improved in the paper of Fox and Sauermann. The proof of Theorem \ref{thm:Main_Theorem} uses a variant of the \emph{
slice rank method} \cite{TaosBlogCapsets} that we call the \emph{partition
rank method} \cite{Naslund2017Partition}. The slice rank method was
introduced by Tao \cite{TaosBlogCapsets} following the work of Ellenberg
and Gijswijt \cite{EllenbergGijswijtCapsets}, and the breakthrough
result of Croot, Lev, and Pach \cite{CrootLevPachZ4}. Our bound for
the Erd\H{o}s-Ginzburg-Ziv constant relies on the more general notion
of the partition rank, as defined in \cite{Naslund2017Partition},
which allows us to handle the indicator tensor that appears when we
force the variables to be distinct. The slice rank method has seen
numerous applications, such as Tao's proof of Ellenberg and Gijswijt's upper bound
for progression-free sets in $\mathbb{F}_{p}^{n}$ \cite{TaosBlogCapsets,EllenbergGijswijtCapsets},
new bounds for the Erd\H{o}s-Szemer\'{e}di sunflower problem \cite{NaslundSawinSunflower},
disproving certain conjectures concerning fast matrix multiplication
\cite{BlasiakChurchCohnGrochowNaslundSawinUmans2016MatrixMultiplication},
and right angles in $\mathbb{F}_{q}^{n}$ \cite{GeShangguan2017NoRightAngles}.
We refer the reader to \cite{BlasiakChurchCohnGrochowNaslundSawinUmans2016MatrixMultiplication}
for a more detailed discussion of the properties of the slice rank,
and it's relationship to geometric invariant theory.

\subsection{Property D}
Theorem \ref{thm:Main_Theorem}  can be substantially improved if the distinctness condition is weakened.
\begin{thm}
\label{thm:q_el_not_distinct} Let $q$ be an odd prime power. Suppose that $A\subset\left(\mathbb{Z}/q\mathbb{Z}\right)^{n}$
satisfies 
\[
|A|>\gamma_{q}^{n}
\]
where 
\[
\gamma_{q}=\min_{0<x<1}\frac{1-x^{q}}{1-x}x^{-\frac{q-1}{q}}.
\]
Then $A$ contains $q$ not necessarily distinct, but not all equal,
elements $x_{1},\dots,x_{q}$ which sum to $0$.
\end{thm}
Note that for any $q$,
\begin{equation}
\gamma_{q}<\inf_{0<x<1}\frac{1}{1-x}x^{-1}=4,\label{eq:gamma_q_bound}
\end{equation}
and so the statement above holds with $\gamma_{q}^{n}$ replaced with
$4^{n}$. Theorem \ref{thm:q_el_not_distinct} implies near optimal bounds for the Erd\H{o}s-Ginzburg-Ziv under a conjecture of Gao and Geroldinger on the maximal counter-example, known as \emph{property D}.

We say that a group $G$ satisfies property $D$ (see \cite[Sec. 7]{GaoGeroldinger2006ZeroSumSurvey})
if every sequence $S$ of length $\mathfrak{s}(G)-1$ that does not
contain $\exp(G)$ elements summing to $0$, when viewed as a multi-set,
takes the form 
\[
S=\cup_{i=1}^{\exp(G)-1}T
\]
for some set $T$. That is, if $S$ is a maximal sequence that does
not contain $\exp(G)$ elements summing to zero, then every element
in $S$ appears exactly $\exp(G)-1$ times. Gao and Geroldinger conjecture
\cite[Conj. 7.2]{GaoGeroldinger2006ZeroSumSurvey} that $\left(\mathbb{Z}/k\mathbb{Z}\right)^{n}$
satisfies property $D$ for every $k$ and $n$. Under this conjecture
we can improve upon Theorem \ref{thm:Main_Theorem}, Corollary \ref{cor: cp_asymp},
and Alon and Dubiner's bound for the Erd\H{o}s-Ginzburg-Ziv constant.
This also improves upon a recent result of Heged\H{u}s \cite{Hegedus2017EGZ}
that used the Croot-Lev-Pach polynomial method to show that if $\mathbb{F}_{p}^{n}$
satisfies property $D$, then 
\[
\mathfrak{s}\left(\mathbb{F}_{p}^{n}\right)\leq(p-1)p^{\left(1-\frac{(p-2)^{2}}{2p^{2}\log p}\right)n+1}+1.
\]

\begin{thm}
\label{thm:property D Theorem} Let $q$ be a prime power, and assume
that $\left(\mathbb{Z}/q\mathbb{Z}\right)^{n}$ satisfies property
$D$. Then 
\[
\mathfrak{s}\left(\left(\mathbb{Z}/q\mathbb{Z}\right)^{n}\right)\leq(q-1)4^{n}+1.
\]
\end{thm}
Using the results of section \ref{sec:Reduction-to-Prime}, Theorem
\ref{thm:property D Theorem} implies the following result for general
$k$:
\begin{thm}
\label{thm:prop_d_for_general_k} Suppose that $k=p_{1}^{r_{1}}\cdots p_{m}^{r_{m}}$,
and 
\[
\left(\mathbb{Z}/p_{i}^{r_{i}}\mathbb{Z}\right)^{n}
\]
satisfies property $D$ for each $1\leq i\leq m$. Then
\[
\mathfrak{s}\left(\left(\mathbb{Z}/k\mathbb{Z}\right)^{n}\right)\leq(k-1)4^{n}+k.
\]
\end{thm}

The proof of Theorem \ref{thm:q_el_not_distinct} uses the slice rank
method, and does not require the added flexibility of the partition
rank. To handle prime powers, we use an idea appearing in \cite{BlasiakChurchCohnGrochowNaslundSawinUmans2016MatrixMultiplication},
where Ellenberg and Gijswijt's bounds for progression-free
sets in $\mathbb{F}_{p}^{n}$ are extended to progression-free sets in $(\mathbb{Z}/k\mathbb{Z})^{n}$
using a binomial coefficient indicator function.

\subsection*{Outline:}
In section \ref{sec:Reduction-to-Prime}, we reproduce the reduction argument of Fox and Sauermann \cite{FoxSauermann2017EGZ}.  In section \ref{sec:Partition} we define the partition-rank which was introduced by the author in \cite{Naslund2017Partition}. In section \ref{sec:Unconditional} we prove Theorem \ref{thm:Main_Theorem} by using the partition-pank method, and bounding from above the partition-rank of a tensor that we construct to capture the combinatorial problem. In section \ref{sec:Conditional} we prove \ref{thm:q_el_not_distinct} using the slice-rank method, and deduce \ref{thm:property D Theorem} and \ref{thm:prop_d_for_general_k} as a consequence.

\section{\label{sec:Reduction-to-Prime}Reduction to Prime Powers}

In this section, we reproduce several Lemmas that provide upper bounds
for $\mathfrak{s}\left(\left(\mathbb{Z}/k\mathbb{Z}\right)^{n}\right)$
based on $\mathfrak{s}\left(\left(\mathbb{Z}/p\mathbb{Z}\right)^{n}\right)$
for the primes $p$ dividing $k$. These Lemmas were first communicated to the author by Sauermann, and later appeared in Fox and Sauermann's paper \cite{{FoxSauermann2017EGZ}}. From them Corollary\ref{cor: cp_asymp} and Theorem \ref{thm:prop_d_for_general_k}
can be deduced from Theorems \ref{thm:Main_Theorem} and \ref{thm:property D Theorem},
respectively.
\begin{lem}
\label{lem:EGZ_Lemma} (\cite[proposition 3.1]{ChiDingGaoGeroldingerSchmid2005ZeroSumSubsequences},
see also \cite[lemma 3.1]{FoxSauermann2017EGZ}) Let $G$ be a finite
abelian group. Let $H\subset G$ be a subgroup such that $\exp(G)=\exp(H)\exp\left(G/H\right)$.
Then 
\[
\mathfrak{s}(G)\leq\exp\left(G/H\right)\left(\mathfrak{s}(H)-1\right)+\mathfrak{s}\left(G/H\right).
\]
\end{lem}
By repeated application of this lemma, Fox and Sauermann gave the
following two lemmas:
\begin{lem}
\label{lem:g_reduced}(\cite[Lemma 3.4]{FoxSauermann2017EGZ}) Let
$G$ be a non-trivial finite abelian group, and suppose that $G\cong G_{1}\times\cdots\times G_{m}$
where the exponents $\exp\left(G_{i}\right)$ are pairwise relatively
prime. Then
\[
\mathfrak{s}\left(G\right)\leq\sum_{i=1}^{m}\exp\left(G_{1}\right)\dotsm\exp\left(G_{i-1}\right)\mathfrak{s}(G_{i}).
\]
\end{lem}
\begin{lem}
\label{lem:k_reduced_to_p}(\cite[Lemma 3.5]{FoxSauermann2017EGZ})
Let $k=p_{1}^{\alpha_{1}}\cdots p_{r}^{\alpha_{r}}$. Then we have
that 
\[
\mathfrak{s}\left((\mathbb{Z}/k\mathbb{Z})^{n}\right)<k\cdot\left(\frac{\mathfrak{s}\left(\mathbb{F}_{p_{1}}^{n}\right)}{p_{1}-1}+\cdots+\frac{\mathfrak{s}\left(\mathbb{F}_{p_{r}}^{n}\right)}{p_{r}-1}\right).
\]
\end{lem}
Lemma \ref{lem:k_reduced_to_p} combined with Theorem \ref{thm:Main_Theorem}
implies that 
\begin{align*}
\mathfrak{s}\left((\mathbb{Z}/k\mathbb{Z})^{n}\right) & <k\cdot\left(2^{p_{1}}(J(p_{1})p_{1})^{n}+\cdots+2^{p_{r}}(J(p_{r})p_{r})^{n}\right)\\
 & \leq kr2^{p}(J(p)p)^{n}\\
 & =(J(p)p+o(1))^{n}
\end{align*}
where $p=\max_{i}p_{i}$, and so Corollary \ref{cor: cp_asymp} follows.
To deduce Theorem \ref{thm:prop_d_for_general_k} from Theorem \ref{thm:property D Theorem},
write 
\[
\left(\mathbb{Z}/k\mathbb{Z}\right)^{n}=\left(\mathbb{Z}/q_{1}\mathbb{Z}\right)^{n}\times\cdots\times\left(\mathbb{Z}/q_{m}\mathbb{Z}\right)^{n}
\]
where $q_{1}\leq\cdots\leq q_{m}$ are pairwise relatively prime prime
powers dividing $k$. By Lemma \ref{lem:g_reduced} and Theorem \ref{thm:property D Theorem},
we have that 
\[
\mathfrak{s}\left((\mathbb{Z}/k\mathbb{Z})^{n}\right)\leq k\cdot\frac{(q_{m}-1)4^{n}+1}{q_{m}}+k\cdot\frac{(q_{m-1}-1)4^{n}+1}{q_{m}q_{m-1}}+\cdots+k\cdot\frac{(q_{1}-1)4^{n}+1}{q_{m}\cdots q_{1}}.
\]
The series telescopes, and we obtain
\[
\mathfrak{s}\left((\mathbb{Z}/k\mathbb{Z})^{n}\right)\leq(k-1)4^{n}+k\left(\frac{1}{q_{m}}+\frac{1}{q_{m}q_{m-1}}+\cdots+\frac{1}{q_{m}\cdots q_{1}}\right).
\]
 Since $q_{i}\geq2$, the sum $\frac{1}{q_{m}}+\frac{1}{q_{m}q_{m-1}}+\cdots+\frac{1}{q_{m}\cdots q_{1}}$
will be at most $1$, and Theorem \ref{thm:prop_d_for_general_k}
follows.

\section{\label{sec:Partition}The Partition Rank}

To motivate the definition of the slice rank and the partition rank
we begin by recalling the definition of the tensor rank. For a field $\mathbb{F}$ and finite
sets $X_{1},\dots,X_{n}$, a non-zero function 
\[
h:X_{1}\times\cdots\times X_{k}\rightarrow\mathbb{F},
\]
is called a \emph{rank 1 function} if there exists $f_{1},\dots,f_{k}$
such that 
\[
h(x_{1},\dots,x_{k})=f_{1}(x_{1})f_{2}(x_{2})\cdots f_{k}(x_{k}),
\]
where for each $i$ 
\[
f_i:X_i \rightarrow \mathbb{F}.
\]
The tensor rank of a function 
\[
F:X_{1}\times\cdots\times X_{k}\rightarrow\mathbb{F}
\]
 is defined to be the minimal $r$ such that 
\[
F=\sum_{i=1}^{r}g_{i}
\]
where the $g_{i}$ are rank 1 functions. This function $F$ gives rise to a tensor when thought of as indexing a multidimensional array over $\mathbb{F}$. Specifically $F$ gives rise to a $|X_{1}|\times\cdots\times|X_{k}|$ array of elements
of $\mathbb{F}$, and the rank $1$ functions correspond to tensors given by the outer
product of $k$ vectors. To define the partition rank, we first need
some notation. Given variables $x_{1},\dots,x_{k}$, and a set $S\subset\{1,\dots,k\}$,
$S=\left\{ s_{1},\dots,s_{m}\right\} $, let $\vec{x}_{S}$ denote
the $m$-tuple 
\[
x_{s_{1}},\dots,x_{s_{m}},
\]
so that for a function $g$ of $m$ variables, we have 
\[
g(\vec{x}_{S})=g(x_{s_{1}},\dots,x_{s_{m}}).
\]
For example, if $k=5$, and $S=\{1,2,4\}$, then $g\left(\vec{x}_{S}\right)=g(x_{1},x_{2},x_{4})$,
and $f\left(\vec{x}_{\{1,\dots,k\}\backslash S}\right)=f(x_{3},x_{5})$.
A partition of $\{1,2,\dots,k\}$ is a collection $P$ of non-empty
pairwise disjoint subsets of $\{1,\dots,k\}$ such that 
\[
\bigcup_{A\in P}A=\{1,\dots,k\}.
\]
We say that $P$ is the \emph{trivial} partition if it consists only
of a single set, $\{1,\dots,k\}$.
\begin{defn} \cite{Naslund2017Partition,TaosBlogCapsets}
Let $X_{1},\dots,X_{k}$ be finite sets, and let 
\[
h:X_{1}\times\cdots\times X_{k}\rightarrow\mathbb{F}.
\]
We say that $h$ has \emph{partition rank} $1$ if there exists some
non-trivial partition $P$ of the variables $\{1,\dots,k\}$ such
that 
\[
h(x_{1},\dots,x_{k})=\prod_{A\in P}f_{A}\left(\vec{x}_{A}\right)
\]
for some functions $f_{A}$. We say that $h$ has \emph{slice rank}
$1$, if in addition one of the sets $A\in P$ is a singleton.
\end{defn}
The function $h:X_{1}\times\cdots\times X_{k}\rightarrow\mathbb{F}$
will have partition rank $1$ if and only if it can be written in
the form 
\[
h(x_{1},\dots,x_{k})=f(\vec{x}_{S})g(\vec{x}_{T})
\]
for some $f,g$ and some $S,T\neq\emptyset$ with $S\cup T=\{1,\dots,k\}$.
Additionally, $h$ will have slice rank $1$ if it can be written
in the above form with either $|S|=1$ or $|T|=1$. In other words,
a function $h$ has partition rank $1$ if the tensor can be written
as a non-trivial outer product, and it has slice rank $1$ if it can
be written as the outer product between a vector and a $k-1$ dimensional
tensor.
\begin{defn}
Let $X_{1},\dots,X_{k}$ be finite sets. The \emph{partition rank}
of
\[
F:X_{1}\times\cdots\times X_{k}\rightarrow\mathbb{F},
\]
is defined to be the minimal $r$ such that 
\[
F=\sum_{i=1}^{r}g_{i}
\]
where the $g_{i}$ have partition rank $1$. The \emph{slice rank}
of $F$ is the minimal $r$ such that 
\[
F=\sum_{i=1}^{r}g_{i}
\]
where the $g_{i}$ have slice rank $1$.
\end{defn}
The slice rank and the usual tensor rank are ranks that arise from different partitions of the $k$-variables. The partition rank is the minimal rank among all possible ranks obtained
from partitioning the variables. The slice rank can be viewed as the
rank which results from the partitions of $\{1,\dots,k\}$ into a
set of size $1$ and a set of size $k-1$, and so we have that 
\[
\prank\leq\slicerank.
\]
For two variables, the slice rank, partition rank, and tensor rank
are equivalent since there is only one non-trivial partition of a
set of size $2$. For three variables, the partition rank and the
slice rank are equivalent, and for $4$ or more variables, all three
ranks are different. A key property of the partition rank is the following
Lemma, given in \cite[Lemma 11]{Naslund2017Partition}, which generalizes
\cite[Lemma 1]{TaosBlogCapsets}.
\begin{lem}
\label{lem:Critical_Lemma}Let $X$ be a finite set, and let $X^{k}$
denote the $k$-fold Cartesian product of $X$ with itself. Suppose
that 
\[
F:X^{k}\rightarrow\mathbb{F}
\]
corresponds to the diagonal identity tensor, that is 
\[
F(x_{1},\dots,x_{k})=\delta(x_{1},\dots,x_{k})=\begin{cases}
1 & x_{1}=\cdots=x_{k}\\
0 & \text{otherwise}
\end{cases}.
\]
Then 
\[
\prank(F)=|X|.
\]
\end{lem}
\begin{proof}
See \cite[Lemma 11]{Naslund2017Partition}.
\end{proof}

\section{\label{sec:Unconditional}Unconditional Bounds for $\mathfrak{s}\left(\mathbb{F}_{p}^{n}\right)$}

We begin with a Lemma that bounds from above the number of monomials
of degree at most $d$ over $\mathbb{F}_{q}^{n}$. This is a standard
application of Markov's inequality.
\begin{lem}
\label{lem:number_of_monomials} Let $q$ be a prime power. As functions\footnote{When viewed as distinct functions, as opposed to abstract polynomials, the identity $x^q-x=0$ implies that the degree of each $x_j$ in each monomial is at most $q-1$.}, the number of monomials of degree
at most $d$ over $\mathbb{F}_{q}^{n}$ equals 
\begin{equation}
\#\left\{ v\in\{0,\dots,q-1\}^{n}:\ \sum_{i=1}^{n}v_{i}\leq d\right\} ,\label{eq:number_monomials}
\end{equation}
and this is at most
\[
\min_{0<x<1}\left(\frac{1-x^{q}}{1-x}x^{-\frac{d}{n}}\right)^{n}.
\]
\end{lem}
\begin{proof}
Let $X_{1},\dots,X_{n}$ denote independent uniform random variables
on $\{0,\dots,q-1\}$, and let $X=\sum_{i=1}^{n}X_{i}$. Then (\ref{eq:number_monomials})
equals
\[
q^{n}\mathbb{P}\left(X\leq d\right),
\]
and by Markov's inequality, for any $0<x<1$
\begin{align*}
q^{n}\mathbb{P}\left(X\leq d\right) & =q^{n}\mathbb{P}\left(x^{d}\leq x^{X}\right)\\
 & \leq q^{n}\mathbb{E}\left(x^{X}\right)x^{-d}\\
 & =q^{n}\left(\mathbb{E}\left(x^{X_{1}}\right)\right)^{n}x^{-d}\\
 & =\left(\frac{1-x^{q}}{1-x}x^{-\frac{d}{n}}\right)^{n}.
\end{align*}
\end{proof}
Let $p>2$ be a prime. Define
\[
F_{p}:\left(\mathbb{F}_{p}^{n}\right)^{p}\rightarrow\mathbb{F}_{p}
\]
by
\begin{equation}
F_{p}(x_{1},\dots,x_{p})=\prod_{i=1}^{n}\left(1-\left(x_{1i}+x_{2i}+\cdots+x_{pi}\right)^{p-1}\right),\label{eq:F_p definition}
\end{equation}
where $x_{ji}$ is understood to be the $i^{th}$ coordinate of the
vector $x_{j}\in\mathbb{F}_{p}^{n}$, and the notation $\left(\mathbb{F}_{p}^{n}\right)^{p}$
is used to denote the $p$-fold Cartesian product of $\mathbb{F}_{p}^{n}$
with itself. Then 
\[
F_{p}\left(x_{1},\dots,x_{p}\right)=\begin{cases}
1 & \text{if }x_{1}+x_{2}+\cdots+x_{p}=0\\
0 & \text{otherwise}
\end{cases}.
\]
This tensor does not take into account whether or not the variables
are distinct, and so in particular for any $x\in\mathbb{F}_{p}^{n}$
\[
F_{p}(x,\dots,x)=1.
\]
In order to modify this tensor so that it picks up only distinct $p$-tuples
of elements summing to zero, as well as the diagonal elements, we use the technique in \cite{Naslund2017Partition},
and introduce an indicator which is a sum over the permutations in
the symmetric group $S_{p}$. More generally, let $X$ be a finite set, and let $k\geq3$ be an integer. For every $\sigma\in S_{k}$, define
\begin{equation}
f_{\sigma}:X^k\rightarrow\mathbb{F}\label{eq:f_sigma def}
\end{equation}
to be the function that is $1$ if $(x_{1},\dots,x_{p})$ is a fixed
point of $\sigma$, and $0$ otherwise. The following is \cite[Lemma 14]{Naslund2017Partition}:
\begin{lem}
\label{lem:Group-action-lemma}We have the identity
\[
\sum_{\sigma\in S_{k}}\sgn(\sigma)f_{\sigma}(x_{1},\dots,x_{k})=\begin{cases}
1 & \text{if }x_{1},\dots,x_{k}\text{ are distinct}\\
0 & \text{otherwise}
\end{cases},
\]
where $\sgn(\sigma)$ is the sign of the permutation\@.
\end{lem}
\begin{proof}
By definition, 
\[
\sum_{\sigma\in S_{k}}\sgn(\sigma)f_{\sigma}(x_{1},\dots,x_{k})=\sum_{\sigma\in\text{Stab}(\vec{x})}\sgn(\sigma)
\]
where $\text{Stab}(\vec{x})\subset S_{k}$ is the stabilizer of $\vec{x}$.
Since the stabilizer is a product of symmetric groups, this will be
non-zero precisely when $\text{Stab}(\vec{x})$ is trivial, and hence
$x_{1},\dots,x_{k}$ must be distinct. This vector is then fixed only
by the identity element, and so the sum equals $1$.
\end{proof}
Using this Lemma, we give the following modification of \cite[Lemma 15]{Naslund2017Partition}:
\begin{lem}
\label{lem:critical-indicator-function}Let $\mathcal{C}_{i}\subset S_{k}$
denote the set all elements in $S_{k}$ which are the product of exactly
$i$ disjoint cycles. Define 
\[
R_{k}(x_{1},\dots,x_{k})=\sum_{\substack{
\sigma\in S_{k}\\
\sigma\notin \mathcal{C}_{1},\mathcal{C}_{2}
}}\sgn(\sigma)f_{\sigma}(x_{1},\dots,x_{k}).
\]
Then for $k\geq3$,
\[
R_{k}(x_{1},\dots,x_{k})=\begin{cases}
1 & \text{if }x_{1},\dots,x_{k}\text{ are distinct}\\
(-1)^{k-1}(k-1)!\sum_{j=2}^{k-1}\frac{1}{j} & \text{if }x_{1}=\cdots=x_{k}\\
\alpha(x_{1},\dots,x_{k}) & \text{if }x_{1},\dots,x_{k}\text{ take on }2\text{ distinct values}\\
0 & \text{otherwise}
\end{cases},
\]
where 
\[
\alpha(x_{1},\dots,x_{k})=(-1)^{k-1}\#\left\{ \sigma\in\mathcal{C}_{2}\text{ that fix }(x_{1},\dots,x_{k})\right\} .
\]
\end{lem}
\begin{proof}
If there are $3$ or more distinct elements among $x_{1},\dots,x_{k}$
then $f_{\sigma}(x_{1},\dots,x_{k})=0$ for any $\sigma\in\mathcal{C}_{1},\mathcal{C}_{2}$,
and so the identity holds by Lemma \ref{lem:Group-action-lemma}.
When there are exactly two distinct elements among $x_{1},\dots,x_{k}$,
$f_{\sigma}(x_{1},\dots,x_{k})=0$ for any $\sigma\in\mathcal{C}_{1}$,
and so 
\[
R_{k}(x_{1},\dots,x_{k})=(-1)^{k-1}\#\left\{ \sigma\in\mathcal{C}_{2}\text{ that fix }(x_{1},\dots,x_{k})\right\} 
\]
since $\sgn(\sigma)=(-1)^{k}$ for any $\sigma\in\mathcal{C}_{2}$.
Lastly, when $x_{1}=\cdots=x_{k}$, we have that 
\[
R_{k}(x_{1},\dots,x_{k})+(-1)^{k}|\mathcal{C}_{2}|+(-1)^{k-1}|\mathcal{C}_{1}|=0
\]
by Lemma \ref{lem:Group-action-lemma} since $\sgn\sigma=(-1)^{k-1}$
for $\sigma\in\mathcal{C}_{1}$. Since
\[
|\mathcal{C}_{1}|=(k-1)!\ \ \ \text{and}\ \ \ |\mathcal{C}_{2}|=(k-1)!\sum_{j=1}^{k-1}\frac{1}{j},
\]
it follows that 
\[
R_{k}(x_{1},\dots,x_{k})=(-1)^{k-1}\sum_{j=2}^{k-1}\frac{1}{j}
\]
when $x_{1}=\cdots=x_{k}$.
\end{proof}

Let 
\begin{equation}
\delta(x_{1},\dots,x_{k})=\begin{cases}
1 & x_{1}=\cdots=x_{k}\\
0 & \text{otherwise}
\end{cases},\label{eq:delta_def}
\end{equation}
and for a partition $P$ of $\{1,\dots,k\}$ define 
\[
\delta_{P}(x_{1},\dots,x_{k})=\prod_{A\in P}\delta(\vec{x}_{A})
\]
where if $A=\{j\}$ is a singleton, $\delta(\vec{x}_{A})$ is defined
to be the constant function $1$. Let $\sigma\in S_{k}$, and consider
its disjoint cycle decomposition. Let $A_{1},\dots,A_{m}$ denote
the sets of indices corresponding to the disjoint cycles. Then $P=\{A_{1},\dots,A_{m}\}$
is a partition of $\{1,\dots,k\}$ and 
\begin{equation}\label{eq:delta_prod}
f_{\sigma}(x_{1},\dots,x_{k})=\delta_{P}(x_{1},\dots,x_{k})=\prod_{i=1}^{m}\delta(\vec{x}_{A_{i}}).
\end{equation}
 It follows that we can write $R_{k}(x_{1},\dots,x_{k})$ exactly
as a linear combination of products of disjoint delta functions, including
the constant function. When $k=3$
\[
R_{3}(x_{1},x_{2},x_{3})=1,
\]
when $k=4$
\begin{align*}
R_{4}(x_{1},x_{2},x_{3},x_{4})= & 1-\delta(x_{1},x_{2})-\delta_{2}(x_{2},x_{3})-\delta(x_{3},x_{4})\\
 & -\delta(x_{4},x_{1})-\delta(x_{1},x_{3})-\delta(x_{2},x_{4}),
\end{align*}
and when $k=5$
\begin{align*}
R_{5}(x_{1},x_{2},x_{3},x_{4},x_{5})= & 1-\sum_{i<j\le5}\delta(x_{i},x_{j})+2\sum_{i<j<l\leq5}\delta(x_{i},x_{j},x_{l})\\
 & + \sum_{i<j\leq5}\delta(x_{i},x_{j})\sum_{\substack{
l<m\leq5\\
l,m\neq i,j
}}\delta(x_{l},x_{m}).
\end{align*}
\begin{rem}Starting with $k=5$, there will be terms that are the product of
multiple delta functions in multiple variables, and to handle these
we need to use the partition rank. For example, as a function on $X^{4}$,
\[
\delta(x_{1},x_{2})\delta(x_{3},x_{4})
\]
has slice rank equal to $|X|$ but partition rank equal to $1$ \cite{TaoSawinBlog}. 
\end{rem}

\begin{lem}
Let $X=\mathbb{F}_p^n$ and let $k=p$. Define 
\[
\tilde{R}_p:\left(\mathbb{F}_p^n\right)^p\rightarrow \mathbb{F}_p 
\]
by taking $R_p$ modulo $p$. Then
\[
\tilde{R}_{p}(x_{1},\dots,x_{p})=\begin{cases}
1 & \text{if }x_{1},\dots,x_{p}\text{ are distinct}\\
1 & \text{if }x_{1}=\cdots=x_{p}\\
\alpha(x_{1},\dots,x_{p}) & \text{if }x_{1},\dots,x_{k}\text{ take on }2\text{ distinct values}\\
0 & \text{otherwise}
\end{cases},
\]
where $\alpha$ was defined in Lemma \ref{lem:critical-indicator-function}.
\end{lem}
\begin{proof}
By pairing every element in $\mathbb{F}_{p}^{\times}$ with its additive inverse
we have that 
\[
\sum_{j=2}^{p-1}\frac{1}{j}=-1+\sum_{j=1}^{p-1}\frac{1}{j}=-1,
\]
and so by Wilson's Theorem
\[
(-1)^{p-1}(p-1)!\sum_{j=2}^{p-1}\frac{1}{j}=(-1)^{p-1}=1.
\]
Hence if $x_{1}=\cdots=x_{p}$ then $\tilde{R}_{p}(x_{1},\dots,x_{p})=1$, and the result follows by Lemma \ref{lem:critical-indicator-function}.
\end{proof}

We now use the function $\tilde{R}_{p}(x_{1},\dots,x_{p})$ to modify $F_{p}$,
and arrive at our desired indicator tensor.
\begin{lem}
\label{lem:I_p def}For $x_{1},\dots,x_{p}\in\mathbb{F}_{p}^{n}$
 define 
\[
I_{p}(x_{1},\dots,x_{p})=\tilde{R}_{p}(x_{1},\dots,x_{p})F_{p}(x_{1},\dots,x_{p})
\]
where $F_{p}$ is defined in (\ref{eq:F_p definition}) and $\tilde{R}_{p}$
is defined above. Then
\[
I_{p}(x_{1},\dots,x_{p})=\begin{cases}
1 & \text{if }x_{1},\dots,x_{p}\text{ are distinct and sum to zero}\\
1 & \text{if }x_{1}=\cdots=x_{p}\\
0 & \text{otherwise}
\end{cases}.
\]
\end{lem}
\begin{proof}

If $x_1,\dots,x_p$ do not sum to zero, then $F_p(x_1,\dots,x_p)=0$. Assume that $x_1+\cdots+x_p=0$. If there are exactly two distinct elements among $x_{1},\dots,x_{p}$,
say $x,y\in\mathbb{F}_{p}^{n}$, then 
\[
x_{1}+\cdots+x_{p}=cx+(p-c)y=c(x-y)
\]
where $1\leq c\leq p-1$. Since $x\neq y$, this cannot equal $0$,
and so $F_{p}(x_{1},\dots,x_{p})=0$ in this case. When $x_{1},\dots,x_{p}$
are distinct and sum to zero, then $F_{p}(x_{1},\dots,x_{p})=1$ and
$R_{p}(x_{1},\dots,x_{p})=1$, and so the result follows. 
\end{proof}
We are now ready to prove Theorem \ref{thm:Main_Theorem}.
\begin{proof}
Suppose that $A\subset\mathbb{F}_{p}^{n}$ does not contain distinct
$x_{1},\dots,x_{p}$ such that 
\[
x_{1}+\cdots+x_{p}=0.
\]
Let $I_{p}$ be defined as in Lemma \ref{lem:I_p def}. When we restrict
$I_{p}$ to $A^{p}$, it will be the diagonal tensor with $1$'s on
the diagonal, and so by Lemma \ref{lem:Critical_Lemma} 
\[
|A|\leq\prank(I_{p}).
\]
 To bound the partition rank of $I_{p}(x_{1},\dots,x_{p})$, we split
$\tilde{R}_{p}$ into a sum of $f_{\sigma}$ terms, and examine each of the resulting terms
\begin{equation}
f_{\sigma}(x_{1},\dots,x_{p})F_{p}(x_{1},\dots,x_{p}).\label{eq:Breaking_up_I_p_equ}
\end{equation}
By equation (\ref{eq:delta_prod}), we may write
\[
f_{\sigma}(x_{1},\dots,x_{p})=\prod_{i=1}^{m}\delta\left(\vec{x}_{S_{i}}\right)
\]
where the $S_{i}$ are disjoint non-empty sets, and where we have
the convention $\delta(\vec{x}_{S})=1$ when $S$ is a singleton set. Since the definition of $\tilde{R}_p$ excluded $\sigma \in\mathcal{C}_1,\mathcal{C}_2$, every term appearing in our sum will have $\sigma\in\mathcal{C}_j$ for $j\geq 3$, and so $f_\sigma$ will be a product of at least three delta functions, and so we have $m\geq 3$.   
For each $1\leq i\leq m$, let $s_{i}$ be the minimal element of $S_{i}$. Then we have
the equality of functions
\[
f_{\sigma}(x_{1},\dots,x_{p})F_{p}(x_{1},\dots,x_{p})=\prod_{i=1}^{m}\delta\left(\vec{x}_{S_{i}}\right)\prod_{j=1}^{n}\left(1-\left(\sum_{i=1}^{m}|S_i|x_{s_{i}j}\right)^{p-1}\right).
\]
 Note that by the disjointness of
the sets $S_{1},\dots,S_{m}$, the variable $x_{s_{i}}$ will not
appear in any delta function $\delta\left(\vec{x}_{S_{j}}\right)$
for $j\neq i$. We may expand the polynomial above as a linear combination
of monomials of the form
\[
\left[\delta\left(\vec{x}_{S_{1}}\right)x_{s_{1}1}^{e_{11}}\cdots x_{s_{1}n}^{e_{1n}}\right]\left[\delta\left(\vec{x}_{S_{2}}\right)x_{s_{2}1}^{e_{21}}\cdots x_{s_{2}n}^{e_{2n}}\right]\cdots\left[\delta\left(\vec{x}_{S_{m}}\right)x_{s_{1}1}^{e_{m1}}\cdots x_{s_{m}n}^{e_{mn}}\right],
\]
where $e_{ij}\leq p-1$ and $\sum_{i=1}^{m}\sum_{j=1}^{n}e_{ij}\leq(p-1)n$.
For each term, there will be some coordinate $i$ such that the monomial
\[
x_{s_{i}1}^{e_{i1}}\cdots x_{s_{i}n}^{e_{in}}
\]
has degree at most $\frac{(p-1)n}{m}$. Futhermore, we must have $m\geq3$
since the permutations $\sigma$ with $1$ or $2$ cycles in their
disjoint cycle decomposition do not appear in $R_{k}$. Now, consider
the simultaneous expansion of all $f_{\sigma}(x_{1},\dots,x_{p})F_{p}(x_{1},\dots,x_{p})$
for each permutation $\sigma$, that is the complete expansion of
$I_{p}(x_{1},\dots,x_{p})$. By the above analysis, for each term
$v(x_{1},\dots,x_{p})$ appearing in the expansion, there exists a
set $S\subset\{1,2,\dots,p\}$, of size $|S|\leq p-2$, such that
for $s = \min S$,
\[
v(x_{1},\dots,x_{p})=\left[\delta\left(\vec{x}_{S}\right)x_{s1}^{e_{s1}}\cdots x_{sn}^{e_{sn}}\right]h\left(\vec{x}_{\{1,\dots,p\}\backslash S}\right)
\]
and 
\[
\sum_{i=1}^{n}e_{si}\leq\frac{(p-1)n}{3},
\]
where $h$ is some function. By grouping terms by delta function with
lowest degree monomial, and counting the number of monomials of degree
at most $\frac{n(p-1)}{3}$, we find that the partition rank of $I_{p}$
is at most
\begin{equation}
\sum_{\substack{
S\subset\{1,\dots,p\}\\
1\leq|S|\leq p-2
}}\#\left\{ v\in\{0,1,\dots,p-1\}^{n}:\ \sum_{i=1}^{n}v_{i}\le\frac{n(p-1)}{3}\right\} ,\label{eq:true_bound}
\end{equation}
and so by Lemma \ref{lem:number_of_monomials} we have the upper bound
\begin{align*}
|A| \leq & \prank(I_{p})  \\ \leq &\left(2^{p}-p-2\right)\left(\min_{0<x<1}\frac{1-x^{p}}{1-x}x^{-\frac{p-1}{3}}\right)^{n}\\
  = &(2^{p}-p-2)(pJ(p))^{n}.
\end{align*}
The bound for $\mathfrak{s}(\mathbb{F}_{p}^{n})$ follows from the
fact that any sequence can have at most $p-1$ copies of the same
element without trivially introducing a solution to 
\[
x_{1}+\cdots+x_{p}=0.
\]
\end{proof}

\section{\label{sec:Conditional}Conditional Bounds for $\mathfrak{s}\left(\left(\mathbb{Z}/q\mathbb{Z}\right)^{n}\right)$}

We first show that Theorem \ref{thm:property D Theorem} follows as a Corollary
of Theorem \ref{thm:q_el_not_distinct} and (\ref{eq:gamma_q_bound}).
\begin{proof}
(Theorem \ref{thm:q_el_not_distinct} implies Theorem \ref{thm:property D Theorem})
Let $S$ be a sequence in $\left(\mathbb{Z}/q\mathbb{Z}\right)^{n}$
of length $s\left(\left(\mathbb{Z}/q\mathbb{Z}\right)^{n}\right)-1$
which does not contain $q$ elements that sum to $0$. By property
$D$, it follows that as a multi-set
\[
S=\cup_{i=1}^{q-1}T
\]
for some $T\subset\left(\mathbb{Z}/q\mathbb{Z}\right)^{n}$. This
implies that the only solution to 
\[
x_{1}+\cdots+x_{q}=0
\]
for $x_{i}\in T$, with no distinctness requirement, occurs when $x_{1}=x_{2}=\cdots=x_{q}$.
From Theorem \ref{thm:q_el_not_distinct} we obtain $|T|\leq\gamma_{q}^{n},$
and hence $\mathfrak{s}\left(\left(\mathbb{Z}/q\mathbb{Z}\right)^{n}\right)\leq(q-1)4^{n}+1$.
\end{proof}
To prove Theorem \ref{thm:q_el_not_distinct} we use the slice rank
method. The following Lemma appears in Theorem 4.14 and Proposition 4.15 in \cite{BlasiakChurchCohnGrochowNaslundSawinUmans2016MatrixMultiplication}:
\begin{lem}
Let $q$ be a prime power. For any $x_{1},\dots,x_{k}\in\mathbb{Z}/q\mathbb{Z}$, 
modulo $p$ we have that
\[
\sum_{m_{1}+\cdots+m_{k}\leq q-1}(-1)^{m_{1}+\cdots+m_{k}}\binom{x_{1}}{m_{1}}\cdots\binom{x_{k}}{m_{k}}=\begin{cases}
1 & \text{if }x_{1}+\cdots+x_{k}=0\\
0 & \text{otherwise}
\end{cases}.
\]
\end{lem}
Note that the binomial coefficient is well defined modulo $p$ due
to Lucas' Theorem.
\begin{proof}
This follows since 
\[
\sum_{m\leq q-1}(-1)^{m}\binom{x_{1}+\cdots+x_{k}}{m}=\begin{cases}
1 & \text{if }x_{1}+\cdots+x_{k}=0\\
0 & \text{otherwise}
\end{cases},
\]
and 
\[
\binom{x_{1}+\cdots+x_{k}}{m}=\sum_{m_{1}+\cdots+m_{k}=m}\binom{x_{1}}{m_{1}}\cdots\binom{x_{k}}{m_{k}}.
\]
\end{proof}
For $x_{1},\dots,x_{k}\in\left(\mathbb{Z}/q\mathbb{Z}\right)^{n}$
let 
\[
E_{q}:\left(\left(\mathbb{Z}/q\mathbb{Z}\right)^{n}\right)^{q}\rightarrow\mathbb{F}_{p}
\]
be defined by
\begin{equation}
E_{q}(x_{1},\dots,x_{q})=\prod_{i=1}^{n}\left(\sum_{m_{1}+\cdots+m_{q}\leq q-1}(-1)^{m_{1}+\cdots+m_{q}}\binom{x_{1i}}{m_{1}}\cdots\binom{x_{qi}}{m_{q}}\right)\label{eq:E_q_def}
\end{equation}
where $x_{ji}$ denotes the $i^{th}$ coordinate of the vector $x_{j}$.
For any $q$ we have that 
\[
E_{q}\left(x_{1},\dots,x_{q}\right)=\begin{cases}
1 & \text{if }x_{1}+x_{2}+\cdots+x_{q}=0\\
0 & \text{otherwise}
\end{cases},
\]
and $E_{q}(x,\dots,x)=1$ for any $x\in\left(\mathbb{Z}/q\mathbb{Z}\right)^{n}$.
\begin{prop} \label{prop:slice_rank}
\label{prop: slice_rank_E_q}The slice-rank of $E_{q}$ on $\left(\mathbb{Z}/q\mathbb{Z}\right)^{n}$
is at most $q\cdot\gamma_{q}^{n}$.
\end{prop}
\begin{proof}
Expanding the product form for $E_{q}\left(x_{1},\dots,x_{q}\right)$
appearing in (\ref{eq:E_q_def}) as a polynomial, we may write $E_{q}\left(x_{1},\dots,x_{q}\right)$
as a linear combination of terms of the form 
\[
\left(\binom{x_{11}}{m_{11}}\cdots\binom{x_{1n}}{m_{1n}}\right)\cdots\left(\binom{x_{q1}}{m_{q1}}\cdots\binom{x_{qn}}{m_{qn}}\right)
\]
where $m_{ij}\leq q-1$ for every $i,j$ and 
\[
\sum_{i=1}^{q}\sum_{j=1}^{n}m_{ij}\leq(q-1)n.
\]
Thus, for each term there is a coordinate $i$ such that 
\[
\sum_{j=1}^{n}m_{ij}\leq\frac{(q-1)n}{q},
\]
and by always slicing away the lowest degree piece, and by grouping terms according to this lowest degree piece, it follows that
\begin{equation}
\slicerank(E_{q})\leq q\cdot\#\left\{ v\in\{0,1,\dots,q-1\}^{n}:\ \sum_{i=1}^{n}v_{i}\le\frac{n(q-1)}{q}\right\} .\label{eq:intermediate_eq_1}
\end{equation}
By Lemma \ref{lem:number_of_monomials} this will be at most 
\[
q\cdot\left(\frac{1-x^{q}}{1-x}x^{-\frac{q-1}{q}}\right)^{n}
\]
and so by Proposition \ref{prop:slice_rank} we conclude that 
\[
\slicerank(E_{q})\leq q\cdot\gamma_{q}^{n}.
\]
\end{proof}
We now prove Theorem \ref{thm:q_el_not_distinct}. 
\begin{proof}
Suppose that $A\subset\left(\mathbb{Z}/q\mathbb{Z}\right)^{n}$ does
not contain a non-trivial $q$-tuple that sums to zero. Then when
restricted to $A^{q}$, $E_{q}$ will be a diagonal tensor taking
the value $1$ on the diagonal. Hence Lemma \ref{lem:Critical_Lemma}
implies that 
\[
|A|\leq\slicerank(E_{q}),
\]
and so 
\[
|A|\leq q\cdot\gamma_{q}^{n}.
\]
The factor of $q$ can be removed by an amplification argument. The
$m$-fold Cartesian product $A^{m}=A\times\cdots\times A$ viewed
as a subset of $\left(\mathbb{Z}/q\mathbb{Z}\right)^{nm}$ will not
contain a non-trivial $k$-tuple summing to $0$, and so by the same
slice rank argument
\[
|A|^{m}\leq q\cdot\gamma_{q}^{nm},
\]
and hence 
\[
|A|\leq q^{\frac{1}{m}}\gamma_{q}^{n}.
\]
Letting $m\rightarrow\infty$, it follows that $|A|\leq\gamma_{q}^{n}$.
\end{proof}

\specialsection*{Acknowledgements}

I would like to thank Will Sawin for his helpful comments and for
his simple proofs of Lemmas \ref{lem:Group-action-lemma} and \ref{lem:critical-indicator-function}, as well as the anonymous referees whose comments and suggestions helped improve the writing and exposition of the paper.
I would also like to thank Lisa Sauermann for her many helpful comments,
corrections, and suggestions, and for pointing out Lemmas \ref{lem:EGZ_Lemma},
\ref{lem:g_reduced}, and \ref{lem:k_reduced_to_p}. This work was
partially supported by the NSERC PGS-D scholarship, and by Ben Green's
ERC Starting Grant 279438, Approximate Algebraic Structure and Applications.

\bibliographystyle{amsplain}

\providecommand{\bysame}{\leavevmode\hbox to3em{\hrulefill}\thinspace}
\providecommand{\MR}{\relax\ifhmode\unskip\space\fi MR }
% \MRhref is called by the amsart/book/proc definition of \MR.
\providecommand{\MRhref}[2]{%
  \href{http://www.ams.org/mathscinet-getitem?mr=#1}{#2}
}
\providecommand{\href}[2]{#2}

\end{document}